\documentclass[12pt]{amsart}

\usepackage{amsthm,amsfonts,amsmath,amssymb,mathrsfs,verbatim,cite}

\numberwithin{equation}{section}
\newtheorem{theorem}{Theorem}[section]

\newtheorem{conjecture}[theorem]{Conjecture}
\newtheorem{lemma}[theorem]{Lemma}

\renewcommand{\and}{\quad\text{and}\quad}

\newcommand{\iup}{\textup{i}}
\newcommand{\A}{\textup{A}}
\DeclareMathOperator*{\RL}{L}

\newcommand{\Nat}{\mathbb N}
\newcommand{\Z}{\mathbb Z}

\newcommand{\Symm}{\mathfrak{S}}
\newcommand{\abs}[1]{\lvert#1\rvert}
\newcommand{\la}{\lambda}
\newcommand{\qbin}[2]{\genfrac{[}{]}{0pt}{}{#1}{#2}}

\newcommand{\g}{\mathfrak{g}}
\newcommand{\h}{\mathfrak{h}}
\newcommand{\n}{\mathfrak{n}}

\newcommand{\vb}{\boldsymbol{v}}
\newcommand{\mb}{\boldsymbol{m}}
\newcommand{\wb}{\boldsymbol{w}}
\newcommand{\ub}{\boldsymbol{u}}
\newcommand{\xb}{\boldsymbol{x}}
\newcommand{\rb}{\boldsymbol{r}}
\newcommand{\rhob}{\boldsymbol{\rho}}
\renewcommand{\Im}{\textup{Im}}
\newcommand{\dup}{\textup{d}}

\begin{document}

\title[Dedekind's $\eta$-function 
and Rogers--Ramanujan identities]
{Dedekind's $\boldsymbol{\eta}$-function and 
Rogers--Ramanujan identities}

\author{S. Ole Warnaar and Wadim Zudilin}
\thanks{Work supported by the Australian Research Council}

\address{School of Mathematics and Physics, 
The University of Queensland, Brisbane, QLD 4072, Australia}
\address{School of Mathematical and Physical Sciences, 
University of Newcastle, Callaghan, NSW 2308, Australia}


\begin{abstract}
We prove a $q$-series identity that generalises Macdonald's 
$\A_{2n}^{(2)}$ $\eta$-function identity and the Rogers--Ra\-ma\-nu\-jan 
identities.  We conjecture our result to generalise even further
to also include the Andrews--Gordon identities.
\end{abstract}

\maketitle

\section{Introduction}
In 1972 Macdonald published his seminal paper \cite{Macdonald72}
in which he extended Weyl's denominator formula for classical
reduced root systems to root systems of affine type.
These identities, which include the Jacobi triple product identity
and the quintuple product identity as the special cases
$\A_1^{(1)}$ and $\A_2^{(2)}$,
are now commonly known as the Macdonald identities. 
Through the procedure of specialisation
the Macdonald identities imply identities for powers of 
the Dedekind $\eta$-function
\[
\eta(\tau)=q^{1/24}\prod_{j=1}^{\infty}(1-q^j),
\]
where $q=\exp(2\pi\iup\tau)$ and $\Im(\tau)>0$.
For example, the specialisation \cite[p. 138, (6)(c)]{Macdonald72}
of the $\A_{2n}^{(2)}$ Macdonald identity corresponds the following 
beautiful generalisation of the Euler pentagonal number theorem:
\begin{equation}\label{eta}
\eta(\tau)^{2n^2-n}=
\sum \xi(\vb/\rhob) 
(-1)^{\abs{\vb}-\abs{\rhob}}
q^{||\vb||^2/(2(2n+1))}.
\end{equation}
Here $\vb=(v_1,\dots,v_n)$, $\rhob=(1/2,3/2,\dots,n-1/2)$,
$\abs{\vb}=v_1+\cdots+v_n$, 
$||\vb||^2=\vb\cdot\vb=v_1^2+\cdots+v_n^2$,
\[
\xi(\vb/\wb)=
\prod_{1\leq i<j\leq n} \frac{v_i^2-v_j^2}{w_i^2-w_j^2}
\]
and the sum on the right of \eqref{eta} is over $\vb\in(\Z/2)^n$
such that $v_i\equiv \rho_i\pmod{2n+1}$.

Another famous family of combinatorial identities are the formulae
of Rogers and Ramanujan \cite{Andrews74b,Rogers94}
\begin{subequations}\label{RR}
\begin{align}\label{RR1}
\sum_{m=0}^{\infty}\frac{q^{m^2}}{(q)_m}&=
\frac{(q^2,q^3,q^5;q^5)_{\infty}}{(q)_{\infty}} \\
\sum_{m=0}^{\infty}\frac{q^{m^2+m}}{(q)_m}&=
\frac{(q,q^4,q^5;q^5)_{\infty}}{(q)_{\infty}}
\label{RR2}
\end{align}
\end{subequations}
and their generalisations to arbitrary odd moduli due to Andrews 
and Gordon \cite{Andrews74,Gordon61}
\begin{equation}\label{AG}
\sum_{m_1,\dots,m_{k-1}} 
\frac{q^{M_1^2+\cdots+M_{k-1}^2+M_p+\cdots+M_{k-1}}}
{(q)_{m_1}\cdots (q)_{m_{k-1}}}
=\frac{(q^p,q^{2k-p+1},q^{2k+1};q^{2k+1})_{\infty}}{(q)_{\infty}},
\end{equation}
where $1\leq p\leq k$ and $M_i=m_i+\cdots+m_{k-1}$.
In \eqref{RR} and \eqref{AG} we employ the standard $q$-notation
\[
(a)_m=(a;q)_m=\prod_{i=1}^m(1-aq^{i-1})
\]
and
\[
(a_1,\dots,a_k)_m=(a_1,\dots,a_k;q)_m=
(a_1;q)_m\cdots(a_k;q)_m
\]
for $m\in\Nat\cup\{\infty\}$ (with the convention that $\Nat=\{0,1,2,\dots\}$).

\medskip
 
In this paper we link
Macdonald's identity \eqref{eta} to the Rogers--Ramanujan and Andrews--Gordon
identities \eqref{RR} and \eqref{AG}. 
More specifically, we present a family of $q$-series identities
depending on positive integers $k$, $n$ and $p\in\{1,k\}$ such that
\begin{enumerate}
\item For $k=1$ we recover Macdonald's $\A_{2n}^{(2)}$ identity \eqref{eta}
for the Dedekind $\eta$-function.
\item For $n=1$ and $k=2$ we recover, modulo the Jacobi triple 
product identity, the Rogers--Ramanujan identities \eqref{RR}.
\item For $n=1$ and general $k$ we recover the $p=k$ and $p=1$ 
instances of the Andrews--Gordon identities \eqref{AG}.
\item For general $n$ and $k\to\infty$ we recover the $\A_{2n-1}$ 
case of an identity of Hua related to Kac's conjectures for 
representations of quivers.
\end{enumerate}

The fact that the Rogers--Ramanujan identities have a close connection
with affine root systems or, more generally, affine Kac--Moody algebras 
is not new,
and well-known are the interpretations of \eqref{RR} and \eqref{AG}
in terms of standard modules of $\A_1^{(1)}$, see e.g.,
\cite{LM78,LW78,LW82,LW84,Misra84} and as characters corresponding to certain
non-unitary Virasoro modules, see e.g., \cite{FS94,RC85}.
For our generalisation of the Rogers--Ramanujan and Andrews--Gordon identities,
however, it is crucial to interpret the right-hand sides of \eqref{RR}
and \eqref{AG} as of type type $\A_2^{(2)}$, not $\A_1^{(1)}$.

\medskip

Before stating our main results we observe that, by an appeal to 
Jacobi's triple product identity \cite[(II.28)]{GR04}, the right-hand side
of \eqref{AG} may be rewritten as
\begin{multline*}
\frac{1}{(q)_{\infty}}\sum_{j=-\infty}^{\infty}
(-1)^j q^{(2k+1)\binom{j}{2}+pj} \\ 
=\frac{q^{-(2k-2p+1)^2/(8(2k+1))}}{(q)_{\infty}}
\sum (-1)^{v-k+p-1/2} q^{v^2/(2(2k+1))},
\end{multline*}
where in the second expression the sum is over 
$v\in\Z/2$ such that $v\equiv k-p+1/2\pmod{2k+1}$.
Comparing the sum on the right with the sum in \eqref{eta} it
takes little imagination to make the following conjecture.

For $1\leq a,b\leq N-1$, let $C_{ab}$ be the Cartan integers of 
the Lie algebra $\A_{N-1}$, i.e., $C_{aa}=2$, $C_{a,a\pm 1}=-1$ and 
$C_{ab}=0$ otherwise. By abuse of notation, for $\wb=(w_1,\dots,w_n)$
and $a$ a scalar, set $\wb+a=(w_1+a,\dots,w_n+a)$ so that, in
particular,
\[
||\wb+a||^2=||\wb||^2+2a\abs{\wb}+na^2
\quad\text{and}\quad
\abs{\wb+a}=\abs{\wb}+na.
\]

\begin{conjecture}\label{conj1}
For $k,n$ positive integers, $N=2n$ and $p\in\{1,k\}$,
\begin{multline}\label{eqc1}
\sum
\frac{q^{\frac{1}{2}\sum_{a,b=1}^{N-1}\sum_{i=1}^{k-1}
C_{ab}M_i^{(a)}M_i^{(b)}+
\sum_{a=1}^{N-1}\sum_{i=p}^{k-1} (-1)^a M_i^{(a)}}}
{\prod_{a=1}^{N-1}\prod_{i=1}^{k-1}(q)_{m_i^{(a)}}} \\
=\frac{1}{(q)_{\infty}^{2n^2-n}}
\sum \xi(\vb/\rhob) (-1)^{\abs{\vb}-\abs{\rhob+k-p}}
q^{\tfrac{||\vb||^2-||\rhob+k-p||^2}{2(2k+2n-1)}},
\end{multline}
where the sum on the left is over $m_i^{(a)}\in\Nat$
\textup{(}for all $1\leq a\leq N-1$ and $1\leq i\leq k-1$\textup{)}
and the sum on the right is over $\vb\in(\Z/2)^n$ such that
$v_i\equiv \rho_i+k-p\pmod{2k+2n-1}$. 
The integers $M_i^{(a)}$ are defined as 
$M_i^{(a)}=m_i^{(a)}+\cdots+m_{k-1}^{(a)}$, i.e.,
$m_i^{(a)}=M_i^{(a)}-M_{i+1}^{(a)}$ for $1\leq i\leq k-2$
and $m_{k-1}^{(a)}=M_{k-1}^{(a)}$.
\end{conjecture}

\begin{theorem}[Generalised Rogers--Ramanujan identities]\label{thm}
Conjecture~\ref{conj1} is true for $k=2$. That is, 
for $n$ a positive integer and $N=2n$,
\begin{equation}\label{RR1n}
\sum_{\mb\in\Nat^{N-1}}
\frac{q^{\frac{1}{2}\mb C \mb^t}}{(q)_{\mb}}
=\frac{1}{(q)_{\infty}^{2n^2-n}}
\sum \xi(\vb/\rhob) (-1)^{\abs{\vb}-\abs{\rhob}}
q^{\tfrac{||\vb||^2-||\rhob||^2}{2(2n+3)}},
\end{equation}
where the sum on the right is over $\vb\in(\Z/2)^n$ such that
$v_i\equiv\rho_i\pmod{2n+3}$,
and
\begin{equation*}
\sum_{\mb\in\Nat^{N-1}}
\frac{q^{\frac{1}{2}\mb C \mb^t+\abs{\mb}_{-}}}{(q)_{\mb}}
=\frac{1}{(q)_{\infty}^{2n^2-n}}
\sum \xi(\vb/\rhob) (-1)^{\abs{\vb}-\abs{\rhob+1}}
q^{\tfrac{||\vb||^2-||\rhob+1||^2}{2(2n+3)}},
\end{equation*}
where the sum on the right is over $\vb\in(\Z/2)^n$ such that
$v_i\equiv \rho_i+1\pmod{2n+3}$.
In the above $(q)_{\mb}=(q)_{m_1}\dots(q)_{m_{N-1}}$,
$C$ is the Cartan matrix of $\A_{N-1}$, i.e.,
\[
\frac{1}{2}\mb C\mb^t=
\sum_{i=1}^{N-1}m_i^2-\sum_{i=1}^{N-2}m_im_{i+1},
\]
and, for $\mb\in\Nat^{N-1}$, 
\[
\abs{\mb}_{-}=\sum_{i=1}^{N-1} (-1)^{i-1} m_i.
\] 
\end{theorem}

Analogous to the above theorem, the left-hand side of \eqref{eqc1}
can be expressed without the use of indices by introducing the 
square matrix $B$ of dimension $d:=(N-1)(k-1)$ given by the Kronecker 
product of the Cartan matrix $C$ of $\A_{N-1}$ and the $(k-1)\times (k-1)$ 
matrix $T^{-1}$ with entries $(T^{-1})_{ij}=\min\{i,j\}$:
\[
B_{ai,bj}=(C\otimes T^{-1})_{ai,bj}=C_{ab}\min\{i,j\}.
\]
For example, the $k=p$ instance of \eqref{eqc1} generalises 
\eqref{RR1n} to
\begin{equation}\label{tensor}
\sum_{\mb\in\Nat^d}
\frac{q^{\frac{1}{2}\mb B \mb^t}}{(q)_{\mb}}
=\frac{1}{(q)_{\infty}^{2n^2-n}}
\sum \xi(\vb/\rhob) (-1)^{\abs{\vb}-\abs{\rhob}}
q^{\tfrac{||\vb||^2-||\rhob||^2}{2(2k+2n-1)}}.
\end{equation}
We should remark that the expression on the left-hand side of
\eqref{tensor} has a representation theoretic interpretation
due to Feigin and Stoyanovsky \cite{FS94}. Their 
interpretation in fact holds for $B=C\otimes T^{-1}$ where $C$ is
a Cartan matrix of any semi-simple simply laced Lie algebra
$\g$. Let $\hat{\g}$ be the (nontwisted) affine counterpart of $\g$
and $V_l$ the level-$l$ vacuum integrable highest weight module
of $\g$ with vacuum vector $v$. 
Then $W_l$ is the space $W_l=U(\hat{\n}_{+})\cdot v_0\subset V_l$,
with $U$ the universal enveloping algebra and 
$\n_{-}\oplus\h\oplus\n_{+}$ the Cartan decomposition of $\g$.
The Feigin--Stoyanovsky formula states that
\[
\textup{Tr}(q^{L_0})|_{W_{k-1}}=
\sum_{\mb\in\Nat^d}
\frac{q^{\frac{1}{2}\mb B \mb^t}}{(q)_{\mb}},
\]
where $d=(k-1)\textup{rank}(\g)$ and $L_0$ the energy operator.

There is another elegant expression for the 
left-hand side of \eqref{eqc1} using notation from the theory of 
Hall--Littlewood polynomials, see \cite{Macdonald95,Warnaar06}. 
Let $\la=(\la_1,\la_2,\dots)$ be a partition and $\la'$
its conjugate. Then the $q$-function 
\[
b_{\la}(q)=\prod_{i\geq 1} (q)_{\la'_i-\la'_{i+1}}
\]
features in the Cauchy identity for the Hall--Littlewood polynomials
$P_{\la}(\xb)=P_{\la}(\xb;q)$ and in the principal specialisation formula 
on an infinite alphabet:
\[
P_{\la}(1,q,q^2,\dots)=\frac{q^{n(\la)}}{b_{\la}(q)},
\]
where 
\begin{equation}\label{n}
n(\la)=\sum_{i\geq 1} (i-1)\la_i=\sum_{i\geq 1}\binom{\la_i'}{2}.
\end{equation}
If we further denote $(\la|\mu)=\sum_{i\geq 1} \la_i'\mu_i'$
then the left-hand side of \eqref{eqc1} for $p=k$ or $p=1$ corresponds to
\begin{equation}\label{HL}
\sum
q^{\frac{1}{2}\sum_{a,b=1}^{N-1}C_{ab}(\la^{(a)}|\la^{(b)})}
\prod_{a=1}^{N-1} \frac{z_a^{\abs{\la^{(a)}}}}{b_{\la^{(a)}}(q)}
\end{equation}
summed over partitions $\la^{(1)},\dots,\la^{(N-1)}$ such that
$\la^{(1)}_1,\dots,\la^{(N-1)}_1\leq k-1$, i.e., such that the largest 
parts of the $\la^{(i)}$ do not exceed $k-1$.
In \eqref{HL} $z_i=1$ for all $i$ if $p=k$ and 
$z_{2i-1}=q$ and $z_{2i}=q^{-1}$ for all $i$ if $p=1$.

The rewriting \eqref{HL} of the left-hand side of \eqref{eqc1}
shows that not only the $k=1$ case of \eqref{eqc1} is
known, but also the limiting case $k\to\infty$ (with $p=k$).
Indeed, in \cite{Hua00} Hua derived a combinatorial identity 
related to Kac's conjectures \cite{Kac83} concerning the number 
of isomorphism classes of absolutely indecomposable representations of 
quivers over finite fields. For finite and tame quivers
(i.e., classical and affine ADE) Kac's conjectures are known to be true,
ensuring that in these cases Hua's combinatorial identities are true as
well. For the finite quiver $\A_{N-1}$ ($N$ not necessarily even)
Hua's identity (corrected in \cite{Fulman01}) is
\[
\sum_{\la^{(1)},\dots,\la^{(N-1)}}
q^{\frac{1}{2}\sum_{a,b=1}^{N-1}C_{ab}(\la^{(a)}|\la^{(b)})}
\prod_{a=1}^{N-1} \frac{z_a^{\abs{\la^{(a)}}}}{b_{\la^{(a)}}(q)}
=\prod_{\alpha\in R_{+}} \frac{1}{(z^{\alpha}q)_{\infty}},
\]
where $R_{+}$ is the set of positive roots of $\A_{N-1}$:
\[
R_{+}=\{\alpha_i+\alpha_{i+1}+\cdots+\alpha_j:~1\leq i\leq j\leq N-1\},
\]
and $z^{\alpha_i+\alpha_{i+1}+\cdots+\alpha_j}=z_iz_{i+1}\cdots z_j$.

\section{Further conjectures}

At first sight it may seem surprising that the left-hand side
of \eqref{eqc1} features the root system $\A_{2n-1}$ instead
of, more simply, $\A_n$. However, this is not that unexpected
in view of the following theorem due
to Feigin and Stoyanovsky \cite{FS94} ($n=1$) and Stoyanovsky 
\cite{Stoyanovsky98} ($n>1$).

\begin{theorem}\label{thmFS}
For $k,n$ positive integers, $N=2n+1$ and $p=k$, 
\begin{multline}\label{eqS}
\sum
\frac{q^{\frac{1}{2}\sum_{a,b=1}^{N-1}\sum_{i=1}^{k-1}
C_{ab}M_i^{(a)}M_i^{(b)}+
\sum_{a=1}^{N-1}\sum_{i=p}^{k-1} (-1)^a M_i^{(a)}}}
{\prod_{a=1}^{N-1}\prod_{i=1}^{k-1}(q)_{m_i^{(a)}}} \\
=\frac{1}{(q)_{\infty}^{2n^2+n}}
\sum \chi(\vb/\rhob^{\ast}) 
q^{\tfrac{||\vb||^2-||\rho^{\ast}+k-p||^2}{4(k+n)}},
\end{multline}
where $\rhob^{\ast}=(1,2,\dots,n)$,
\begin{equation}\label{char}
\chi(\vb/\wb)=
\prod_{i=1}^n \frac{v_i}{w_i}
\prod_{1\leq i<j\leq n} \frac{v_i^2-v_j^2}{w_i^2-w_j^2}
\end{equation}
and the sum on the right is over $\vb\in\Z^n$ such that
$v_i\equiv \rho_i^{\ast}+k-p\pmod{2k+2n}$.
\end{theorem}

The Feigin--Stoyanovsky theorem generalises Macdonald's C$_n^{(1)}$
$\eta$-function identity \cite[p. 136, (6)]{Macdonald72}:
\[
\eta(\tau)^{2n^2+n}=
\sum \chi(\vb/\rhob^{\ast}) q^{||\vb||^2/(4(n+1))},
\]
which, for $n=1$, is equivalent to Jacobi's well-known
\[
(q)_{\infty}^3=\sum_{m=0}^{\infty} (-1)^m (2m+1)q^{\binom{m+1}{2}}.
\]

\begin{conjecture}\label{conp1}
Equation \eqref{eqS} is true for $p=1$.
\end{conjecture}

\begin{theorem}\label{thmp1k2}
Conjecture~\ref{conp1} is true for $k=2$, i.e.,
for $n$ a positive integer and $N=2n+1$,
\begin{equation*}
\sum_{\mb\in\Nat^{N-1}}
\frac{q^{\frac{1}{2}\mb C \mb^t+\abs{\mb}_{-}}}{(q)_{\mb}}
=\frac{1}{(q)_{\infty}^{2n^2+n}}
\sum \chi(\vb/\rhob^{\ast}) 
q^{\tfrac{||\vb||^2-||\rhob^{\ast}+1||^2}{4(n+2)}},
\end{equation*}
where the sum on the right is over $\vb\in \Z^n$ such that
$v_i\equiv \rho_i^{\ast}+1\pmod{2n+4}$.
\end{theorem}
Of course, by the Feigin--Stoyanovsky theorem we also have
\begin{equation*}
\sum_{\mb\in\Nat^{N-1}}
\frac{q^{\frac{1}{2}\mb C \mb^t}}{(q)_{\mb}}
=\frac{1}{(q)_{\infty}^{2n^2+n}}
\sum \chi(\vb/\rhob^{\ast}) 
q^{\tfrac{||\vb||^2-||\rhob^{\ast}||^2}{4(n+2)}},
\end{equation*}
where the sum on the right is over $\vb\in \Z^n$ such that
$v_i\equiv \rho_i^{\ast}\pmod{2n+4}$.

\medskip

There is a well-known even-modulus counterpart of the Andrews--Gordon
identities \eqref{AG} due to Bressoud \cite{Bressoud79}:
\begin{equation}\label{Bressoud}
\sum_{m_1,\dots,m_{k-1}} 
\frac{q^{M_1^2+\cdots+M_{k-1}^2+M_p+\cdots+M_{k-1}}}
{(q)_{m_1}\cdots (q)_{m_{k-2}}(q^2;q^2)_{m_{k-1}}}
=\frac{(q^p,q^{2k-p},q^{2k};q^{2k})_{\infty}}{(q)_{\infty}},
\end{equation}
where $k>1$ and, again, $1\leq p\leq k$ and $M_i=m_i+\cdots+m_{k-1}$.
It will be convenient to interpret $1/(q^2;q^2)_{m_0}$ as
$(q)_{\infty}/(q^2;q^2)_{\infty}$ so that \eqref{Bressoud} is true 
for all positive integers $k$.

Our next conjecture unifies Bressoud's identity for $p=k$ with Macdonald's
$\eta$-function identity for $\A_{2n-1}^{(2)}$
\cite[p. 136, (6)(b)]{Macdonald72}:
\begin{equation}\label{A22odd}
\frac{\eta(\tau)^{2n^2+n-1}}{\eta(2\tau)^{2n-1}}
=\sum
\xi(\vb/\rhob^{\star}) (-1)^{(\abs{\vb}-\abs{\rhob^{\star}})/(2n)} 
q^{||\vb||^2/(4n)},
\end{equation}
where $\rhob^{\star}=(0,1,\dots,n-1)$ and $\vb$ is summed over
$\Z^n$ such that $v_i\equiv \rho^{\star}_i\pmod{2n}$,
and a second $\eta$-function identity for $\A_{2n}^{(2)}$
\cite[p. 138, (6)(a)]{Macdonald72}:
\begin{equation}\label{A22even}
\frac{\eta(\tau)^{2n^2+3n}}{\eta(2\tau)^{2n}}
=\sum
\chi(\vb/\rhob) q^{||\vb||^2/(2(2n+1))},
\end{equation}
where the sum is over $\vb\in\Z^n$ such that $v_i\equiv
\rho_i\pmod{2n+1}$.

\begin{conjecture}\label{ConB}
For $k,N$ positive integers and $n=\lfloor N/2 \rfloor$, 
\begin{subequations}\label{conjB}
\begin{align} \notag 
\sum & 
\frac{q^{\frac{1}{2}\sum_{a,b=1}^{N-1}\sum_{i=1}^{k-1}
C_{ab}M_i^{(a)}M_i^{(b)}}}
{\prod_{a=1}^{N-1}\bigl(\prod_{i=1}^{k-2}(q)_{m_i^{(a)}}\bigr)
(q^2;q^2)_{m_{k-1}^{(a)}}} \\
&=\frac{1}{(q)_{\infty}^{N(N-1)/2}}
\sum \xi(\vb/\rhob^{\star}) 
(-1)^{\tfrac{\abs{\vb}-\abs{\rhob^{\star}}}{2k+N-2}}
q^{\tfrac{||\vb||^2-||\rhob^{\star}||^2}{2(2k+N-2)}}
\label{a} \\
&=\frac{1}{(q)_{\infty}^{N(N-1)/2}}
\sum
\chi(\vb/\rhob) q^{\tfrac{||\vb||^2-||\rhob||^2}{2(2k+N-2)}},
\label{b}
\end{align}
\end{subequations}
where \eqref{a} applies for even $N$, in which case the sum is over
$\vb\in\Z^n$ such that $v_i\equiv \rho_i^{\star}\pmod{2k+N-2}$
and \eqref{b} applies for odd $N$, in which case the sum is over
$\vb\in\Z^n$ such that $v_i\equiv \rho_i\pmod{2k+N-2}$.
\end{conjecture}
As before, to recover \eqref{A22odd} and \eqref{A22even} as the $k=1$ case 
of \eqref{conjB} we have to interpret $1/(q^2;q^2)_{m_0^{(a)}}$ as 
$(q)_{\infty}/(q^2;q^2)_{\infty}$.

\section{Dilogarithm identities}

To provide further support for Conjecture~\ref{conj1}, we show below 
that a standard asymptotic analysis applied to \eqref{eqc1} implies an
identity for the Rogers dilogarithm due to Kirillov.

We begin by recalling the definition of the Rogers dilogarithm function 
\[
\RL(x)=-\frac{1}{2}\int_0^x\biggl(\frac{\log(1-t)}{t}+\frac{\log t}{1-t}
\biggr)\dup t, \qquad x\in[0,1].
\]
Note in particular that $\RL(1)=\pi^2/6$.

In \cite{Kirillov87} Kirillov proved the following $\A_{N-1}$ type
dilogarithm identity
\begin{equation}\label{Kir}
\frac{1}{\RL(1)}
\sum_{a=1}^{N-1}\sum_{i=1}^{K-1} \RL\Biggl(
\frac{\sin\bigl(\frac{a\pi}{K+N-1}\bigr)
\sin\bigl(\frac{(N-a)\pi}{K+N-1}\bigr)}
{\sin\bigl(\frac{(i+a)\pi}{K+N-1}\bigr)
\sin\bigl(\frac{(i+N-a)\pi}{K+N-1}\bigr)}
\Biggr)
=\frac{(N^2-1)(K-1)}{K+N-1}.
\end{equation}
If we denote the summand on the left by $S(K,N;a,i)$ then
\[
S(2k,N;a,i)=S(2k,N;a,2k-i-1)\quad\text{for $1\leq i\leq k-1$}
\]
and $S(2k,N;a,2k-1)=\RL(1)$.
Hence
\begin{equation}\label{dil}
\frac{1}{\RL(1)}
\sum_{a=1}^{N-1}\sum_{i=1}^{k-1} \RL\Biggl(
\frac{\sin\bigl(\frac{a\pi}{2k+N-1}\bigr)
\sin\bigl(\frac{(N-a)\pi}{2k+N-1}\bigr)}
{\sin\bigl(\frac{(i+a)\pi}{2k+N-1}\bigr)
\sin\bigl(\frac{(i+N-a)\pi}{2k+N-1}\bigr)}
\Biggr)
=\frac{N(N-1)(k-1)}{2k+N-1}.
\end{equation}

We now recall the following result from \cite{Kirillov95}.
\begin{lemma}
Let $B$ be a $d\times d$ symmetric, positive definite, rational matrix
and let 
\[
\sum_{i=0}^{\infty} a_i q^i=
\sum_{\mb\in\Nat^d}
\frac{q^{\frac{1}{2}\mb B \mb^t}}{(q)_{\mb}}.
\]
Then
\[
\lim_{m\to\infty} \frac{\log^2 a_m}{4m}=\sum_{i=1}^d \RL(x_i),
\]
where the $x_i$ for $1\leq i\leq d$ are the solutions of
\[
x_i=\prod_{j=1}^d(1-x_j)^{B_{ij}}
\]
such that $x_i\in(0,1)$ for all $i$.
\end{lemma}

If we apply the above lemma to the expression on the left-hand side
of \eqref{tensor} we are led to the system of equations
\[
f_i^{(a)}=\prod_{b=1}^{N-1}\prod_{j=1}^{k-1}(1-f_i^{(b)})^{C_{ab}\min\{i,j\}},
\]
for $1\leq a\leq N-1$ and $1\leq i\leq k-1$.
It is readily verified that this is solved by
\[
f_i^{(a)}=\frac{\sin\bigl(\frac{a\pi}{2k+N-1}\bigr)
\sin\bigl(\frac{(N-a)\pi}{2k+N-1}\bigr)}
{\sin\bigl(\frac{(i+a)\pi}{2k+N-1}\bigr)
\sin\bigl(\frac{(i+N-a)\pi}{2k+N-1}\bigr)}.
\]
Hence, denoting the $q$-series on either side of \eqref{tensor}
by $\sum_{i\geq 0} a_i q^i$, we find that
\[
\frac{1}{\RL(1)}\lim_{m\to\infty} 
\frac{\log^2 a_m}{4m}=\text{LHS}\eqref{dil}.
\]
The right-hand side of \eqref{tensor} is a specialised standard
module of $\A_{2n}^{(2)}$ \cite{Kac74,Kac90}. Exploiting
its modular properties \cite{KP84} we obtain
\[
\frac{1}{\RL(1)}\lim_{m\to\infty} 
\frac{\log^2 a_m}{4m}=\text{RHS}\eqref{dil}
\]
(recall that $N=2n$), leading to \eqref{dil}.

In much the same way an asymptotic analysis of Theorem~\ref{thmFS} 
gives \eqref{dil} for odd $N$. The asymptotics of
Conjecture~\ref{ConB}, on the other hand, can be related to the
odd $K$ case of \eqref{Kir}.

\section{Proof of Theorems~\ref{thm} and \ref{thmp1k2}}

The proof given below uses the theory of Hall--Littlewood polynomials,
and for notation and definitions pertaining to these functions
we refer the reader to \cite{Macdonald95,Warnaar06}.

For $\xb=(x_1,\dots,x_n)$ and $\mu$ a partition of length $l(\mu)\leq n$,
let $Q'_{\mu}(\xb)=Q'_{\mu}(\xb;q)$ be the 
modified Hall--Littlewood polynomial \cite{LLT97}
\[
Q'_{\mu}(\xb)=\sum_{\la} K_{\la\mu}(q)s_{\la}(\xb),
\]
where the $K_{\la\mu}(q)$ are the Kostka--Foulkes polynomials and
$s_{\la}(\xb)$ the Schur functions.
In the ring of symmetric functions, the polynomials $Q'_{\mu}$ form 
the adjoint basis of  $P_{\la}$ with respect to the Hall inner product.
They may also be viewed in $\la$-ring notation \cite{Lascoux03} as
$Q'_{\mu}(\xb)=b_{\mu}(q) P_{\mu}(\xb/(1-q))$.
\begin{theorem}\label{thmA}
Let $C$ be the $\A_n$ Cartan matrix.
Then
\begin{subequations}\label{modA}
\begin{equation}\label{mod1}
\sum_{m=0}^{\infty}\frac{q^m}{(q)_m}\, Q'_{(2^m)}(1^n)
=\sum_{\rb\in\Nat^n}
\frac{q^{\frac{1}{2}\rb C\rb^t}}{(q)_{\rb}}
\end{equation}
and
\begin{equation}\label{mod2}
\sum_{m=0}^{\infty}\frac{q^{2m}}{(q)_m}\, 
Q'_{(2^m)}(\underbrace{1,q^{-1},1,q^{-1},\dots}_{n \text{ terms}})
=\sum_{\rb\in\Nat^n}
\frac{q^{\frac{1}{2}\rb C\rb^t+\abs{\rb}_{-}}}{(q)_{\rb}}.
\end{equation}
\end{subequations}
\end{theorem}

Our proof requires a generalisation of \cite{Macdonald95}
\begin{equation}\label{spec}
Q'_{\la}(1)=q^{n(\la)}
\end{equation}
due to Lascoux.
Extend \eqref{n} to skew shapes by
\[
n(\la/\mu)=\sum_{i\geq 1}\binom{\la_i'-\mu_i'}{2}.
\]
\begin{theorem}[\!{\cite[Theorem 3.1]{Lascoux05}}]
For $\mu\subseteq\la$
\[
Q'_{\la/\mu}(1)=
\frac{q^{n(\la/\mu)}}{b_{\mu}(q)}
\prod_{i=1}^{l(\mu)} (1-q^{\la'_{\mu_i}-i+1}),
\]
and $Q'_{\la/\mu}(1)=0$ otherwise.
\end{theorem}
Before we show how this implies Theorem~\ref{thmA}
we note that the above may be rewritten as
\begin{equation}\label{qb}
Q'_{\la/\mu}(1)=q^{n(\la/\mu)} \prod_{i\geq 1}
\qbin{\la_i'-\mu_{i+1}'}{\la_i'-\mu_i'},
\end{equation}
where $\qbin{m}{k}$ is a $q$-binomial coefficient.
A classical result in the theory of abelian $p$-groups
states that if $\alpha_{\la}(\mu;p)$ is the number
of subgroups of type $\mu$ in a finite abelian $p$-group of type 
$\la$, then \cite{Butler94,Delsarte48,Dyubyuk48,Yeh48}
\[
q^{n(\la)-n(\mu)}\alpha_{\la}(\mu;q^{-1})=q^{n(\la/\mu)}
\prod_{i\geq 1} \qbin{\la_i'-\mu_{i+1}'}{\la_i'-\mu_i'}.
\]
This obviously implies that
\[
Q'_{\la/\mu}(1)=
q^{n(\la)-n(\mu)} \alpha_{\la}(\mu;q^{-1}),
\]
a result we failed to find in the literature. 

\begin{proof}[Proof of Theorems~\ref{thmA}]
For $n=1$ \eqref{modA} is trivial since, by \eqref{spec},
\[
Q'_{(2^m)}(1)=q^{m^2-m}.
\]

To prove \eqref{modA} for general $n$ we apply
\[
Q'_{\la}(\xb)=
\sum \prod_{i=1}^n x_i^{\abs{\mu^{(i-1)}-\mu^{(i)}}}
Q'_{\mu^{(i-1)}/\mu^{(i)}}(1),
\]
where the sum is over 
\begin{equation}\label{chain}
0=\mu^{(n)}\subseteq\dots\subseteq\mu^{(1)}\subseteq\mu^{(0)}=\la.
\end{equation}
Hence
\begin{align*}
\text{LHS}\eqref{mod1}
&=\sum_{m=0}^{\infty} \sum
\frac{q^m}{(q)_m}\, 
\prod_{i=1}^n Q'_{\mu^{(i-1)}/\mu^{(i)}}(1) \\
\intertext{and}
\text{LHS}\eqref{mod2}
&=\sum_{m=0}^{\infty} \sum
\frac{q^{2m}}{(q)_m}\, 
\prod_{i=1}^n q^{(-1)^i \abs{\mu^{(i)}}} Q'_{\mu^{(i-1)}/\mu^{(i)}}(1),
\end{align*}
where the inner sums on the right are over \eqref{chain} with $\la=(2^m)$.
If we now change the above double-sums to a sum over
$k_1,\dots,k_{n-1}$ and $\rb=(r_1,\dots,r_n)$ by setting
\[
\mu^{(i)}=(1^{k_i}2^{r_{i+1}+\cdots+r_n-k_i-\cdots-k_{n-1}}),
\quad 0\leq i\leq n,
\]
(where $k_0=k_n:=0$), and insert \eqref{qb}, we arrive at
\begin{align*}
\text{LHS}\eqref{mod1}&
=\sum_{\rb\in\Nat^n}^{\infty} \frac{q^{||\rb||^2}}{(q)_{r_1}}
\prod_{i=1}^{n-1}\sum_{k_i=0}^{\min\{r_i,r_{i+1}\}}
\frac{q^{k_i(k_i-r_i-r_{i+1})}}{(q)_{r_{i+1}-k_i}}\qbin{r_i}{k_i}
\intertext{and}
\text{LHS}\eqref{mod2}&
=\sum_{\rb\in\Nat^n}^{\infty} \frac{q^{||\rb||^2+\abs{\rb}_{-}}}{(q)_{r_1}}
\prod_{i=1}^{n-1}\sum_{k_i=0}^{\min\{r_i,r_{i+1}\}}
\frac{q^{k_i(k_i-r_i-r_{i+1})}}{(q)_{r_{i+1}-k_i}}\qbin{r_i}{k_i}.
\end{align*}
By the $q$-Chu--Vandermonde sum \cite[(II.6)]{GR04} the sum over $k_i$ yields 
\[
\frac{q^{-r_ir_{i+1}}}{(q)_{r_{i+1}}},
\]
thus proving \eqref{modA}.
\end{proof}

Theorem~\ref{thmA} combined with Milne's C$_n$ analogue of the 
Rogers--Selberg identity implies Theorems~\ref{thm} and \ref{thmp1k2} 
as outlined below.

Let $\Delta(\xb)$ be the C$_n$ Vandermonde product
\[
\Delta(\xb)=
\prod_{1\leq i<j\leq n}
(1-x_i/x_j) \prod_{1\leq i\leq j\leq n}
(1-x_ix_j)
\]
and, for $\ub\in\Nat^n$, let $n(\ub)=\sum_{i=1}^n(i-1)u_i$.
\begin{theorem}[\!{\cite[Corollary 2.21]{Milne94}}]\label{thmMilne}
\begin{multline*}
\sum_{\ub\in\Nat^n}
\frac{\Delta(\xb q^{\ub})}{\Delta(\xb)}
\prod_{i,j=1}^n \frac{(x_ix_j)_{u_i}}{(qx_i/x_j)_{u_i}} \\
\times
(-1)^{n\abs{\ub}} q^{n(\ub)+\frac{1}{2}(n+4)||\ub||^2-\frac{1}{2}n\abs{\ub}}
\prod_{i=1}^n x_i^{(n+4)u_i-\abs{\ub}}  \\
=\prod_{i=1}^n (qx_i^2)_{\infty}\prod_{1\leq i<j\leq n}(qx_ix_j)_{\infty}
\sum_{m=0}^{\infty}\frac{q^m}{(q)_m}\, Q'_{(2^m)}(\xb).
\end{multline*}
\end{theorem}

For $\ub\in\Nat^n$ denote by $\ub^{+}$ the unique partition in the
$\Symm_n$ orbit of $\ub$. For example, if $\ub=(2,0,3,2,0)$ then 
$\ub^{+}=(3,2,2,0,0)=(3,2,2)$.
Similarly, let $W_n$ be the Weyl group of C$_n$, i.e., 
$W_n=(\Z_2)^n\rtimes\Symm_n$,
acting on $\vb\in\Z^n$ by permutation and sign-reversal of its
components.
Then $\vb^{\ast}$ denotes the unique partition in the $W_n$ orbit of $\vb$.
For example, if $\vb=(-2,0,-3,2,0)$ then 
$\vb^{\ast}=(3,2,2,0,0)=(3,2,2)$.

Let us now denote the summand on the left of Theorem~\ref{thmMilne}
by $L_n(\ub,\xb)$ and let us denote the right-hand sides of
\eqref{eqc1} and \eqref{eqS} for $k=p=2$ by $R_{2n-1}(\vb)$ 
and $R_{2n}(\vb)$, respectively.
Then, for $\la$ a partition of length at most $n$,
\[
\lim_{\xb\to (1^n)} 
\sum_{\substack{\ub\in\Nat^n \\ \ub^{+}=\la}} L_n(\ub,\xb)
=
\chi\bigl(l(\la)\leq \lfloor (n+1)/2 \rfloor\bigr)
\sum_{\substack{\vb\in\Z^n \\ \vb^{\ast}=\la}} R_n(\vb).
\]
Here $\chi$ is the truth function (and not the character \eqref{char}).
As an immediate consequence of the above we find that
\[
\lim_{\xb\to (1^n)} 
\sum_{\ub\in\Nat^n} L_n(\ub,\xb)
=
\sum_{\vb\in\Z^n} R_n(\vb).
\]
By Theorem~\ref{thmMilne}
\begin{align*}
\lim_{\xb\to (1^n)} &
\sum_{\ub\in\Nat^n} L_n(\ub,\xb) \\
&=
\lim_{\xb\to (1^n)} 
\prod_{i=1}^n (qx_i^2)_{\infty}\prod_{1\leq i<j\leq n}(qx_ix_j)_{\infty}
\sum_{m=0}^{\infty}\frac{q^m}{(q)_m}\, Q'_{(2^m)}(\xb)\\
&=
(q)_{\infty}^{n(n+1)/2}\sum_{m=0}^{\infty}\frac{q^m}{(q)_m}\, Q'_{(2^m)}(1^n).
\end{align*}
Thanks to \eqref{mod1} this proves the $k=p=2$ instances of
\eqref{eqc1} and \eqref{eqS}.

In much the same way one obtains \eqref{eqc1} and \eqref{eqS} with $k=2$ 
and $p=1$ by taking the $\xb\to (q^{1/2},q^{-1/2},q^{1/2},\dots)$ limit 
in Theorem~\ref{thmMilne}.

\bibliographystyle{amsplain}

\end{document}